%% file: elliptic_surfaces.tex
\documentclass[fleqn, 3p]{elsarticle}
\pdfoutput=1

\usepackage[utf8]{inputenc}

\usepackage[english]{babel} 

\usepackage{import}
\usepackage{xifthen}
\usepackage{pdfpages}
\usepackage{transparent}
\usepackage{wrapfig}

\usepackage{graphicx}
\usepackage{subcaption}
\usepackage{amsmath}
\usepackage{amsfonts}

\usepackage[mode=build]{standalone}
\usepackage{tikz}
\usetikzlibrary{cd}
\usetikzlibrary{calc}
\usetikzlibrary{decorations.markings}
\tikzstyle{state}=[
rectangle,
minimum size =1.25cm,
thick,
align=center
]
\usepackage{scalefnt}
\usepackage{hyperref}
\usepackage{cleveref}

\usepackage{xcolor, soulutf8}
\definecolor{bubbles}{rgb}{0.91, 1.0, 1.0}
\definecolor{aquamarine}{rgb}{0.5, 1.0, 0.83}
\definecolor{bubblegum}{rgb}{0.99, 0.76, 0.8}
\definecolor{bluebell}{rgb}{0.64, 0.64, 0.82}
\definecolor{dollarbill}{rgb}{0.72, 0.93, 0.6}

\usepackage{makecell}

\usepackage{amsthm}
\newtheorem{theorem}{Theorem}
\newtheorem{proposition}{Proposition}
\newtheorem{definition}{Definition}
\newtheorem{remark}[theorem]{Remark}
\newtheorem{lemma}[theorem]{Lemma}

\newcommand{\figref}[1]{Fig.~\ref{#1}}
\newcommand{\tabref}[1]{Table~\ref{#1}}
\newcommand{\thmref}[1]{Theorem~\ref{#1}}

\newcommand{\lemref}[1]{Lemma~\ref{#1}}

\newcommand{\secref}[1]{Section~\ref{#1}}
\newcommand{\algoref}[1]{Algorithm~\ref{#1}}

\NewDocumentCommand{\textcite}{om}{\IfValueTF{#1}{\cite[#1]{#2}} {\cite{#2}}}
\NewDocumentCommand{\parencite}{om}{\IfValueTF{#1}{\cite[#1]{#2}} {\cite{#2}}}

\newcommand{\Z}{\mathbb Z}
\newcommand{\N}{\mathbb N}
\newcommand{\Q}{\mathbb Q}
\newcommand{\Qbar}{\bar{\mathbb Q}}
\newcommand{\C}{\mathbb C}
\newcommand{\Proj}{\mathbb P}
\newcommand{\SL}{\operatorname{SL}}
\newcommand{\GL}{\operatorname{GL}}
\newcommand{\im}{ \operatorname{im}}
\newcommand{\res}{\operatorname{Res}}
\newcommand{\NS}{\operatorname{NS}}

\newcommand{\MW}{\operatorname{MW}}
\newcommand{\Triv}{\operatorname{Triv}}
\newcommand{\Prim}{\operatorname{Prim}}

\newcommand{\T}{\mathcal{T}}

\newcommand{\ud}{\mathrm{d}}

\usepackage{algorithm}
\usepackage[noend]{algpseudocode}

\usepackage{booktabs}

\makeatletter 
\def\ps@pprintTitle{%
  \let\@oddhead\@empty
  \let\@evenhead\@empty
  \def\@oddfoot{\reset@font\hfil\thepage\hfil}
  \let\@evenfoot\@oddfoot
}
\makeatother

\usepackage{pseudo}
\pseudoset{
  compact, kw,
  label=\scriptsize\sffamily\color{gray}\arabic*,
  font=\sffamily,
  ct-left=\quad\ctfont{[},
  ct-right=\ctfont{]},
}

\title{A semi-numerical algorithm for\\ the homology lattice and periods of complex elliptic surfaces over $\Proj^1$\tnoteref{t1} }
  \tnotetext[t1]{
     This article is part of the volume titled ``Computational Algebra and Geometry: A special issue in memory and honor of Agnes Szanto''.}
     
\author{Eric Pichon-Pharabod }
\date{\today}

\begin{document}

\begin{abstract}
  We provide an algorithm for computing a basis of homology of elliptic surfaces over $\mathbb P_\C^1$ that is sufficiently explicit for integration of periods to be carried out.
  This allows the heuristic recovery of several algebraic invariants of the surface, notably the Néron--Severi lattice, the transcendental lattice, the Mordell--Weil group and the Mordell--Weil lattice.
  This algorithm comes with a SageMath implementation.
\end{abstract}

\maketitle


\setcounter{tocdepth}{1}
\tableofcontents

\section{Introduction}

An elliptic surface is a surface $S$ together with a proper map $f\colon S\to C$ to a curve, such that the generic fibre is an elliptic curve.
Such surfaces benefit from a rich structure, bridging the elegance of elliptic curves with properties of higher-dimensional varieties.
As such, they exhibit numerous algebraic and topological phenomena.

The set of rational points of an elliptic curve benefits from a group structure, called the Mordell--Weil group.
This group's structure and properties have far-reaching implications in various fields, such as algebraic number theory with the Birch and Swinnerton--Dyer conjecture \parencite{BirchSwinnerton-Dyer_1965}; cryptography with elliptic curve cryptography \parencite{Miller1986, Koblitz1987}; or the study of Diophantine equations \parencite{Lang1978}.
When the elliptic curve is defined over a number field, the Mordell--Weil theorem, a central result in this context, establishes that this group is finitely generated \parencite{Mordell1922, Weil1928}.
Determining the Mordell--Weil group, and in particular determining its rank, however, remains a notoriously difficult problem \parencite{hindry2007}.

Viewing elliptic surfaces as elliptic curves over the function field of the base space reveals a map between rational points of the curve and sections of the fibration.
In this context the Mordell--Weil group gains additional structure, induced from the lattice structure of the second homology group of the surface.
Furthermore in this setting, Shioda provided an isomorphism between the Mordell--Weil group of $S$ and a group derived from the Néron--Severi lattice of the surface \parencite{Shioda1990}.
This isomorphism not only simplifies the computation of the group's rank, yielding the Shioda--Tate formula, but also imparts a lattice structure to the Mordell--Weil group.

In light of these observations, we provide in this paper a semi-numerical algorithm for efficiently computing the homology of elliptic surfaces, as well as high precision numerical approximations of its holomorphic periods. In turn, we obtain a heuristic method for computing their Mordell--Weil group and lattice. All in all, this yields practical computational methods to fields where elliptic surfaces naturally arise, such as the study of Feynman integrals \parencite{BlochKerrVanhove2018, Bonishetal2021, doran2023motivic} and mirror symmetry \parencite{CoxKatz1999, Horietal2003}.

\subsection*{Notations}
Throughout this text, for a smooth variety $X$ defined over a field $k$, the notation $H_n(X)$ will designate the $n$-th singular homology group with integer coefficients of $X$, while $H^n(X)$ will designate its $n$-th algebraic De Rham cohomology group with coefficients in $k$.

\subsection*{Contributions}
We provide an algorithm for computing the full homology lattice, with its intersection product, of an elliptic surface $S$ over the projective line $\Proj^1$. 
The input data is its defining equation as a polynomial $P_t\in \mathbb Q[t][X,Y,Z]$ that is homogeneous of degree $3$ in the variables $X$, $Y$, $Z$ so that the generic fibre defines a cubic in $\Proj^2$. 
When $S$ is not isotrivial, the algorithm also provides:
\begin{itemize}
\item the holomorphic period map $H^{2,0}(S)\times H_2(S)\to \C$ given numerically with certified precision bounds in quasilinear time with respect to precision;
\item an embedding of the Néron--Severi lattice in $H_2(S)$, obtained heuristically as the kernel of the holomorphic period map;
\item the Mordell--Weil group, with the lattice structure of its torsion free part, also depending on the heuristic of the Néron--Severi lattice.
\end{itemize}
By ``heuristically'', we mean that it is possible that the algorithm misses elements of the kernel or finds fake ones, in the case of numerical coincidence --- we provide more details at the end of \secref{sec:periods}.
The algorithm presented here is implemented in SageMath, as part of the package \mbox{\emph{lefschetz-family}}\footnote{\url{https://github.com/ericpipha/lefschetz-family}}.

\subsection*{Previous works}
Algebraic curves coincide with Riemann surfaces, and computation of their periods have been well studied \parencite{DeconinckVanHoeij2001,Swierczewski2017,BruinSijslingZotine2019,MolinNeurohr2019,Neurohr2018}.
Computations of the periods of some varieties realised as double cover ramified along hyperplane arrangements have been carried out in two distinct cases:
in \textcite{CynkvanStraten2019} in the case of Calabi-Yau threefolds given as double covers of $\Proj^3$ ramified along $8$ planes;
and in \textcite{ElsenhansJahnel2022} in the case of K3 surfaces given as a double cover of $\Proj^2$ ramified along $6$ lines.

A method for arbitrary dimensions was first given by \textcite{Sertoz2018} in the case of hypersurfaces.
A computationally cheaper and more general approach was developed in \textcite{LairezPichonVanhove23}.
The work presented in this paper is an improvement on the methods of \textcite{LairezPichonVanhove23}, adapted to the case of elliptic surfaces.

\subsection*{Outline}
We consider elliptic surfaces with a section.
In the case where the elliptic surface only has simple singular fibres --- i.e., singular fibres of type $I_1$ --- semi-numerical methods for computing the homology and periods were developed in \textcite{LairezPichonVanhove23}. 
One may always reduce to this case by deforming the elliptic surface in a way that separates its singular fibres into $I_1$ fibres, thus obtaining a Lefschetz fibration.
By continuity, the homology of this deformation is the same as that of the initial surface.
As the monodromy representation of the deformation only depends on the type of the singular fibres, this can be done formally, without having to consider an explicit realisation of such a morsification.

This provides a way to compute an effective basis of the full homology lattice of an elliptic surface.
Certain cycles can be expressed as \emph{extensions}, i.e., lifts of paths in~$\Proj^1$ avoiding the critical values to~$H_1(S_t)$.
The lattice generated by these cycles and the fibre components is a full rank sub-lattice of ``primary'' cycles $\Prim(S/\Proj^1)\subset H_2(S)$.
Numerical approximations of the holomorphic periods of primary cycles can be computed through numerical integration.
We may thus recover a numerical approximation of the full holomorphic period mapping~$H^{2,0}(S)\times H_2(S)\to \C$.
In turn, we may use the LLL algorithm to heuristically recover the Néron--Severi group in our basis of homology.
This allows the computation of several algebraic invariants of the fibration, notably the Mordell--Weil group, and the lattice associated to its torsion-free part.

\section{Elliptic surfaces}\label{sec:elliptic_surfaces}

In this section, we recall the definition of an elliptic surface, as well as related notions that are useful to our discussion, notably the action of monodromy and Kodaira's classification of singular fibres. For further reading on elliptic surfaces, we recommend \textcite{SchuttShioda10, Miranda1989}.

\begin{definition}
Let $V$ be a complex curve.
An {\em elliptic surface over $V$} is a complex surface $S$ along with a proper surjective map $f\colon S\to V$ such that
\begin{itemize}
\item for all but finitely many $t\in V$, the fibre $F_t = f^{-1}(t)$ is a smooth genus $1$ complex curve (i.e., an elliptic curve);
\item no fibre contains a smooth rational curve of self-intersection $-1$.
\end{itemize}
\end{definition}
The second condition is there to ensure that the surface is relatively minimal, as such rational curves can always be blown down.
We will use the shorthand $S/\Proj^1$ to designate the surface $S$ together with the map $S\to \Proj^1$.
We denote by $\Sigma$ the finite set of values $t\in V$ over which the fibre $F_t $ is not a smooth elliptic curve.

In the following, we consider an elliptic surface $f\colon S\to \Proj^1$ over $V = \Proj^1$.
\begin{definition}
A {\em section} of $S/\Proj^1$ is a map $\pi\colon  \Proj^1\to S$ such that $f\circ \pi = \operatorname{id}_{V}$.
\end{definition}

Sections of $S/\Proj^1$ are in bijection with $\C(t)$-points on the generic fibre $E$, which is an elliptic curve over $\C(t)$.
Indeed, given a section $\pi$, the intersection of its image with the generic fibre $\im\pi \cap E$ yields a point in $E$.
Conversely, given a point $P\in E(\C(t))$, its specialisation to any smooth fibre yields a point of the fiber. 
The closure $\Gamma$ of the union of all these points yields a birational morphism $f|_\Gamma\colon \Gamma\to \Proj^1$. 
The inverse of this map gives the section associated to $P$, which we denote $\bar P$.

Throughout this paper, we will only consider elliptic surfaces with a section.
We will notably fix a section $O$ of $S/\Proj^1$, which we call the \emph{zero section}.
It will serve as the zero of the group of rational points $E(\C(t))$ of the generic elliptic curve.
Furthermore, to avoid the trivial case of a product $E\times\Proj^1$, we require in the rest of this text that the elliptic surface has at least one singular fibre.

\subsection{Monodromy and extensions}
We briefly recall the notions of monodromy and extensions, following \textcite{lamotke} and Section 2.1.2 of \textcite{LairezPichonVanhove23}.
The restriction of $f$ to $f^{-1}(\Proj^1\setminus \Sigma)$ is a locally trivial fibration:
if $U\subset \Proj^1\setminus\Sigma$ is open and simply connected, there is a trivialisation $f^{-1}(U) \simeq F_b\times U$ of the fibration, for all $b\in U$.
In particular a path $\ell \colon  [0,1]\to \Proj^1\setminus \Sigma$ induces a diffeomorphism $F_{\ell(0)}\simeq F_{\ell(1)}$ which is unique up to some automorphism of $F_{\ell(1)}$ that is isotopic to the identity. 
Thus $\ell$ induces an isomorphism $\ell_*\colon H_1(F_{\ell(0)})\to H_1(F_{\ell(1)})$.
For $\ell$, $\ell'$ two such paths compatible for concatenation, one may show the composition formula 
\begin{equation}\label{eq:comp_monodromy}
(\ell'\ell)_* = \ell'_*\circ \ell_*\,,
\end{equation}
(where $\ell'\ell$ is the path that goes through $\ell$ first, then through $\ell'$). 
Furthermore, one may show that the map $\ell_*$ depends only on the homotopy class of $\ell$.

Let $b\in\Proj^1\setminus \Sigma$.
The above construction yields a map 
\begin{equation}
\begin{cases}\pi_1(\Proj^1\setminus\Sigma, b)\to \operatorname{Aut}(H_1(F_b))\\
 [\ell]\mapsto \ell_*
 \end{cases}\,,
\end{equation}
where $[\ell]$ denotes the homotopy class of $\ell$.
The map $\ell_*$ is called the {\em action of monodromy along $\ell$ on $H_1(F_b)$}.

As $F_b$ is an elliptic curve, its first homology group $H_1(F_b)$ is a lattice of rank $2$, equipped with a skew-symmetric intersection product.
As readily seen from the trivialisation of the fibration, monodromy preserves the intersection product.
Therefore, a simple computation shows that the matrix of the action of monodromy in a symplectic basis of $H_1(F_b)$ belongs to $\SL_2(\Z)$.
It is possible to compute this matrix with semi-numerical computations involving the Picard--Fuchs equation of $F_t$, see \textcite[\S3.5.2]{LairezPichonVanhove23}.
\\

A related notion to monodromy is that of {\em extensions}. 
Given a non-intersecting path $\ell$, a simply connected neighbourhood $V$ of $\im\ell$ and a $1$-chain $\Delta$ of $F_{\ell(0)}$, 
the identification of $\Delta\times \im\ell$ in $f^{-1}(V)\subset S$ produces a $2$-chain with boundary in $F_{\ell(0)}\cup F_{\ell(1)}$.
Once again, the relative homology class of this $2$-chain is well-defined in the relative homology group $H_2(S, F_{\ell(0)}\cup F_{\ell(1)})$, and only depends on the homotopy class of $\ell$ and the homology class of $\Delta$.
Therefore we define the {\em extension map} $\tau_{\ell}\colon  H_1(F_{\ell(0)}) \to H_2(S, F_{\ell(0)}\cup F_{\ell(1)})$.
Extensions satisfy a composition rule which allows to extend their definition to self-intersecting paths:
\begin{equation}\label{eq:extensions}
\tau_{\ell'\ell}(\gamma) = \tau_\ell(\gamma) + \tau_{\ell'}(\ell_*(\gamma)) \,,
\end{equation}
in $H_2(S, F_{\ell(0)}\cup F_{\ell(1)}\cup F_{\ell'(1)})$.
In particular, when $\ell$ is a loop pointed at $b$, we obtain a map
\begin{equation}
\tau\colon  \pi_1(\Proj^1\setminus\Sigma, b)\times H_1(F_b) \to H_2(S, F_b) : ([\ell], \Delta) \mapsto \tau_\ell(\Delta)\,.
\end{equation}

Extensions and monodromy are closely related by the formula
\begin{equation}\label{eq:boundary_extension}
\delta(\tau_\ell(\gamma)) = \ell_*\gamma - \gamma\,,
\end{equation}
where $\delta\colon  H_2(S, F_b)\to H_1(F_b)$ is the boundary map. This is represented in \figref{fig:thimbles}

\subsection{The Kodaira classification}
Kodaira provides a classification of the singular fibres of an elliptic fibration \textcite{Kodaira63}:
Let $\sigma\in\Sigma$ be a critical value and $\ell$ be the counter-clockwise simple loop around $\sigma$ pointed at $b$.
As stated in the previous section, the monodromy action $\ell_*$ is represented by a matrix $M\in \SL_2(\Z)$ in a symplectic basis of $H_1(F_b)$.
The $\SL_2(\Z)$-conjugacy class of $M$ determines the type of the singular fibre.

These conjugacy classes are classified in two infinite families $I_\nu$ and $I_\nu^*$, $\nu\in \N$ and six classes $II$, $III$, $IV$, $II^*$, $III^*$ and $IV^*$.
Representatives $M_T$ of these conjugacy classes are given in Table 1 of \textcite{CadavidVelez08}\footnote{There is a typo in Table 1 of \textcite{CadavidVelez08}: the m.n.f. for types $\textit{II}^*$ and $\textit{IV}^*$ are swapped.} (which is reproduced in \tabref{tab:KodairaClassification}), along with a factorisation as a product of $I_1$-type monodromy matrices which will prove useful in \secref{sec:local_morsification}, and the Euler characteristic of the fibre.
These singular fibres have been extensively studied --- for further reading on this topic, we recommend \textcite[Chap. 7]{Esole2017}.

\begin{table}[tp]
\begin{center}
\begin{tabular}{ ccccc } 
 \toprule
Type &$M_T$& \thead{Minimal normal\\  factorisation} & \thead{Euler characteristic\\ of the fibre}\\ 
 \midrule
  $\mathit{I}_\nu, \nu\ge 1$&$\left(\begin{array}{cc} 1& \nu \\ 0 & 1 \end{array}\right)$ & $U^\nu$& $\nu$ \\ 
  \addlinespace[5pt]
  $\mathit{II}$&$\left(\begin{array}{cc} 1& 1 \\ -1 & 0 \end{array}\right)$& $VU$& $2$\\ 
    \addlinespace[5pt]
  $\mathit{III}$& $\left(\begin{array}{cc} 0& 1 \\ -1 & 0 \end{array}\right)$& $VUV$& $3$ \\ 
    \addlinespace[5pt]
  $\mathit{IV}$& $\left(\begin{array}{cc} 0& 1 \\ -1 & -1 \end{array}\right)$& $(VU)^2$& $4$\\ 
    \addlinespace[5pt]
  $\mathit{I}_\nu^*, \nu\ge 0$&$\left(\begin{array}{cc} -1& -\nu \\ 0 & -1 \end{array}\right)$& $U^\nu(VU)^3$& $\nu+6$\\ 
    \addlinespace[5pt]
  $\mathit{II}^*$&$\left(\begin{array}{cc} 0& -1 \\ 1 & 1 \end{array}\right)$& $(VU)^5$& $10$\\ 
    \addlinespace[5pt]
  $\mathit{III}^*$&$\left(\begin{array}{cc} 0& -1 \\ 1 & 0 \end{array}\right)$& $VUV(VU)^3$& $9$\\ 
    \addlinespace[5pt]
  $\mathit{IV}^*$ &$\left(\begin{array}{cc} -1&-1 \\ 1 & 0 \end{array}\right)$& $(VU)^4$& $8$\\ 
    \addlinespace[5pt]
   \bottomrule
     \addlinespace[5pt]
\end{tabular}
\caption{The singular fibre types of the Kodaira classification, representatives of their $\SL_2(\Z)$ conjugacy class of the monodromy matrix, and the minimal normal factorisation of this representative in terms of the $I_1$-type matrices $U = \left(\begin{smallmatrix} 1& 1 \\ 0 & 1 \end{smallmatrix}\right)$ and $V = \left(\begin{smallmatrix} 1& 0 \\ -1 & 1 \end{smallmatrix}\right)$. This table is a reproduction of \textcite[Table 1]{CadavidVelez08}.}
\label{tab:KodairaClassification}
\end{center}
\end{table}

\section{Homology and periods of elliptic surfaces}

In this section, we discuss means to recover the full homology lattice from the knowledge of the mono\-dromy matrices of an elliptic fibration.
Let us first fix some notations.
Let $f\colon S\to \Proj^1$ be an elliptic fibration.
\begin{itemize}
\item $F_v$ denotes the fibre above $v\in\Proj^1$. We will also use this notation for the homology class of $F_v$ in $H_2(S)$ when there is no ambiguity;
\item $O$ denotes the zero section of $S/\Proj^1$;
\item $c_1, \dots, c_r\in \Proj^1$ denote the critical values of $f$, and $\Sigma = \{c_1, \dots, c_r\}$;
\item $m_v$ denotes the number of irreducible components of the fibre $F_v$;
\item $\Theta_0^v$ denotes the {\em zero component} of $F_v$, i.e., the irreducible component intersecting the zero section.
\item $\Theta_1^v, \dots, \Theta_{m_v-1}^v$ denote the irreducible components of the singular fibre $F_v$ that are disjoint from $O$;
\item $\T \subset H_2(S, F_b)$ denotes the image of $\tau$. For a basis $\ell_1, \dots, \ell_{r-1}$ of~${\pi_1(\Proj^1\setminus\Sigma, b)}$, we have $\T = \bigoplus_{i=1}^{r-1}\im\tau_{\ell_i}$ as a direct consequence of \eqref{eq:extensions};
\end{itemize}

We begin by a simple lemma connecting the homology groups $H_2(S)$ and $H_2(S, F_b)$.
\begin{lemma}\label{lem:quotient_kerd}
Let $\delta\colon  H_2(S, F_b)\to H_1(F_b)$ be the boundary map. We have
\begin{equation}
H_2(S)/\langle F_b\rangle \simeq \ker\delta\,.
\end{equation}
\end{lemma}
\begin{proof}
This is a direct consequence of the long exact sequence of relative homology of the pair $(S, F_b)$:
\begin{equation}\label{eq:les_rel_hom}
\langle F_b\rangle = H_2(F_b)\to H_2(S)\to H_2(S, F_b) \overset{\delta}{\to} H_1(F_b)\,.
\end{equation}
\end{proof}

Several cycles in $H_2(S)$ are distinguished in that their associated holomorphic periods (that is, periods of holomorphic forms) are directly computable, see \secref{sec:periods} below. 
In the next paragraph, we define a lattice $\Prim(S/\Proj^1)\subset H_2(S)$ of such cycles.
When the fibration is Lefschetz, $\Prim(S/\Proj^1)$ coincides with the full homology lattice $H_2(S)$.
In general, however, $\Prim(S/\Proj^1)$ may be a proper sublattice of $H_2(S)$.
Nevertheless, we will show in \secref{sec:periods} that $\Prim(S/\Proj^1)$ always has full rank.
In particular, all the periods of $S$ can be recovered from the periods of $\Prim(S/\Proj^1)$.

More precisely, the periods associated to components of fibres and the section are $0$.
Furthermore, \textcite{LairezPichonVanhove23} provides a way to compute the periods of extensions.
We call such cycles {\em primary} and define $\Prim(S/\Proj^1)$ as follows:
\begin{definition}
The {\em primary lattice} $\Prim(S/\Proj^1)$ is the sublattice of $H_2(S)$ generated by extensions, fibre components and the zero section:
\begin{equation}
\Prim(S/\Proj^1) = \phi^{-1}(\T)\oplus \langle O\rangle \oplus \left(O^\perp\cap H_2(\pi^{-1}(\Sigma))\right) \,,
\end{equation}
where $\phi$ is the map $H_2(S)\to H_2(S, F_b)$.
\end{definition}
Note that it follows from the long exact sequence \eqref{eq:les_rel_hom} that the first term is isomorphic to $(\T\cap\ker\delta) \oplus \langle F_b\rangle$. 
Furthermore $\pi^{-1}(\Sigma)$ is the disjoint union of the singular fibres. 
Thus a basis of the last term is given by the $\Theta_i^v$'s where $v$ ranges over $\Sigma$ and $i$ from $1$ to $m_v-1$.

\subsection{The Lefschetz case}\label{sec:lefschetz}

When all the singular fibres are of type $I_1$, $f$ is said to be a {\em Lefschetz fibration} of $S$.
In this setting, \textcite{LairezPichonVanhove23} provides an algorithm for computing an effective basis of $H_2(S)$ from the list of the monodromy matrices of a certain basis of the homotopy group $\pi_1(\C\setminus\Sigma)$ (where $\C\simeq \Proj^1\setminus\{\infty\}$ with $\infty$ a regular value). 
In particular, we have the following theorem.
\begin{theorem}\label{thm:primary_lefschetz}
When $f\colon S\to \Proj^1$ is a Lefschetz fibration, $\Prim(S/\Proj^1) = H_2(S)$.
\end{theorem}
In order to prove this theorem, we first recall the elements of the construction in \textcite{LairezPichonVanhove23} that are relevant to our discussion.
Let $V\subset \Proj^1$ be diffeomorphic to a disk, let $S^* = f^{-1}(V)$, let $\Sigma_V = \Sigma\cap V$, $s=|\Sigma_V|$ and pick $b\in V\setminus\Sigma_V$.

\begin{figure}[tp]
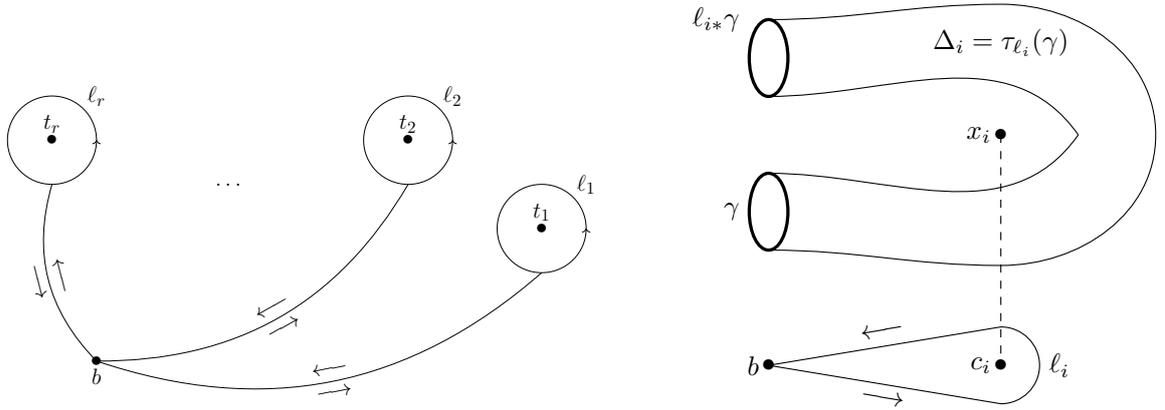

  \centering
  \begin{subfigure}[t]{0.50\linewidth}
    \scalefont{1.5}
  \resizebox{\columnwidth}{!}{\includestandalone{figs/homotopy_basis/homotopy_basis}}

    \caption{A distinguished basis. The composition $\ell_1, \dots, \ell_r$ is represented by the loop encircling all the critical values once counterclockwise.}
    \label{fig:distinguished_basis}
  \end{subfigure}
  \hspace{1em}
    \scalefont{1}
  \begin{subfigure}[t]{0.4\linewidth}
  \resizebox{\columnwidth}{!}{\includestandalone{figs/thimbles/thimbles}}

  \caption{The thimble $\Delta_i\in H_2(S, F_b)$ above a critical value $c_i$ is the nontrivial extension of a $1$-cycle $\gamma\in H_1(F_b)$ along a loop $\ell_i$ around a unique critical value $c_i$. For a given such loop $\ell_i$, there is a unique thimble (up to sign). Its boundary is $\delta\Delta_i = {\ell_i}_{*}\gamma -\gamma$.}
  \label{fig:thimbles}
  \end{subfigure}
  \caption{Distinguished bases and thimbles.}
\end{figure}

\begin{definition}
A basis $\ell_1, \dots, \ell_s$ of $\pi_1(V\setminus\Sigma_V, b)$ is {\em distinguished} if for every $i$, $\ell_i$ is isotopic to a simple (i.e., with winding number $1$) counterclockwise loop around a critical value of $\Sigma_V$, and the composition $\ell_r\cdots\ell_1$ is isotopic to a simple counterclockwise loop around all the points of $\Sigma_V$. This is represented in \figref{fig:distinguished_basis}.
\end{definition}
If $\ell_1, \dots, \ell_s$ is a distinguished basis of $\pi_1(V\setminus\Sigma, b)$, then for each $i$, $\im \tau_{\ell_i} \subset H_2(S^*, F_b)$ has rank $1$ --- its generator (up to sign) is called the {\em thimble} above $c_i$, denoted $\Delta_i$. 
The boundary of $\Delta_i$ is called the {\em vanishing cycle} at $c_i$ and is a generator of the image of ${\ell_i}_*-\operatorname{id}$, as can be readily seen from \eqref{eq:boundary_extension}. 
An illustration of a thimble is given in \figref{fig:thimbles}.

Thimbles serve as building blocks for $H_2(S^*)$, as the following lemma demonstrates.
\begin{lemma}[{\textcite[Main lemma]{lamotke}}]\label{lem:lamotke_main_lemma}
\begin{equation}
H_2(S^*, F_b) = \bigoplus_{i=1}^s\Z\Delta_i
\hspace{1em} \text{ and } \hspace{1em}
H_1(S^*, F_b)=0\,.
\end{equation}
\end{lemma}

In particular when $V = \Proj^1\setminus\{\infty\}$ for a regular point $\infty\notin \Sigma$, the following lemmas show how to recover $H_2(S)$ from $H_2(S^*, F_b)$.

\begin{lemma}
$\ker\left(\iota\colon  H_2(S^*, F_b)\to H_2(S, F_b)\right) = \im \tau_{\infty}$, where $\tau_\infty = \tau_{\ell_r\cdots\ell_1}$.
In other words, an element of $H_2(S^*, F_b)$ is trivial in $H_2(S, F_b)$ if and only if it is a sum of extensions along isotopy classes of $\pi_1(V\setminus\Sigma, b)$ that are trivial in $\pi_1(\Proj^1\setminus\Sigma, b)$.
\end{lemma}
\begin{proof}
This is the second line of diagram (16) in \textcite{LairezPichonVanhove23}.
\end{proof}

\begin{lemma}\label{lem:extensions_and_section}
There is a split short exact sequence
\begin{equation}
0 \to \T\to H_2(S, F_b)\to H_0(F_b)\to 0\,,
\end{equation}
where the first map is the inclusion and the second map is the intersection with the generic fibre $F_b$.
In other words, $H_2(S, F_b) = \T \oplus \langle O\rangle$.
\end{lemma}
\begin{proof}
We have the long exact sequence of the triplet $(S, S^*, F_b)$:
\begin{equation}
H_2(S^*, F_b) \to H_2(S, F_b)\to H_2(S, S^*)\to H_1(S^*, F_b)\,.
\end{equation}
It follows from \lemref{lem:lamotke_main_lemma} that the image of the first map is $\T$.
As $(S, S^*)\simeq F_b\times (D, S^1)$, the Künneth formula yields 
\begin{equation}
H_2(S, S^*)\simeq H_0(F_b)\,,
\end{equation}
where the identification is the intersection product with the generic fibre $F_b$.
Finally, from \lemref{lem:lamotke_main_lemma}, $H_1(S^*, F_b)=0$.
The sequence splits because $H_0(F_b)$ is free.
\end{proof}
\begin{proof}[Proof of \thmref{thm:primary_lefschetz}]
The proposition follows from \lemref{lem:extensions_and_section} and \lemref{lem:quotient_kerd}.
\end{proof}

\subsection{Morsification}

\begin{figure}[tp]
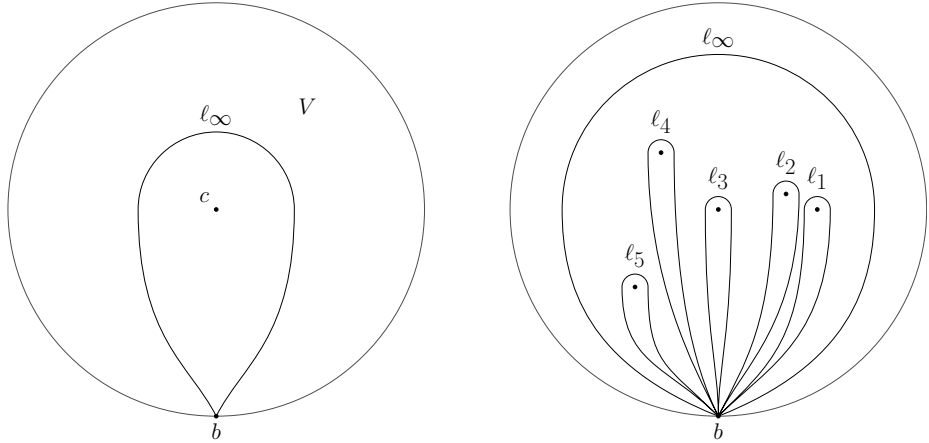

  \centering
  \begin{subfigure}[t]{0.35\linewidth}
    \scalefont{4}
  \resizebox{\columnwidth}{!}{\includestandalone{figs/morsification_2/morsification_2}}

  \end{subfigure}
  \hspace{1em}
    \scalefont{4}
  \begin{subfigure}[t]{0.35\linewidth}
  \resizebox{\columnwidth}{!}{\includestandalone{figs/morsification/morsification}}

  \end{subfigure}
  \vspace{-2em} 
  \caption{The morsification of a neighbourhood of a single critical value. {\em Left:} A neighbourhood $V$ of a single critical value of the elliptic fibration $f$, along with a chosen basepoint $b$.
  The homotopy group $\pi_1(V\setminus\{c\}, b)$ is generated by the counterclockwise loop $\ell_\infty$. 
  {\em Right:} 
  The neighbourhood after morsification. 
  The critical fibre has split into five Lefschetz fibres. 
  The homotopy group $\pi_1(V\setminus \Sigma)$ is generated by the $5$ counterclockwise loops $\ell_1, \dots, \ell_5$. 
  Thus $H_2(f^{-1}(V), F_b)$ has rank $5$ --- it follows from \tabref{tab:KodairaClassification} that the original singular fibre was of type $I_5$.
  Furthermore we see that $\tau_{\ell_\infty} = \tau_{\ell_5\dots \ell_1}$.}
  \label{fig:morsification}
\end{figure}

We have seen in the previous section that the data of the monodromy matrices of the elliptic surface is sufficient to recover the homology in the Lefschetz case.
In this section we extend this result to general elliptic surfaces.

In a nutshell, it is possible to deform the elliptic surface in a way that splits the non-Lefschetz singular fibres into several Lefschetz fibres.
Therefore it is possible to compute the homology of the deformed elliptic surface from its monodromy matrices, which is by construction diffeomorphic to the original one.
Interestingly, the monodromy matrices of the Lefschetz fibres of the deformation are determined by the monodromy matrices of the initial surface.
In particular, this implies that we do not have to do any computations with (or even find) an explicit realisation of such a deformation.
Instead it only serves as a formal computational tool.

More precisely, let $V\subset \Proj^1$ be open and $D = \{z\in\C \mid |z|\le1\}$ denote the unit disk.
\begin{definition}
Let $S\to V$ be an elliptic surface. 
Consider a commutative diagram of proper surjective holomorphic maps between complex manifolds
\begin{equation}
    \begin{tikzcd} 
      \tilde S \arrow[r, "\tilde f"] \arrow[rd, "\eta"]& V\times D \arrow[d, "p"]\\
      & D
    \end{tikzcd}\,,
  \end{equation}
  where $p$ is the projection onto the second coordinate.
For $u\in D$, denote $\tilde S_u = \eta^{-1}(u)$ and $f_u = \tilde f|_{\tilde S_u}\colon \tilde S_u\to V$ where we identify $V\times\{u\}\simeq V$. 
  Such a diagram is a {\em morsification of  $S\to V$} if
  \begin{itemize}
  \item $\eta: \tilde S\to D$ has no critical values;
  \item $f_0: \tilde S_0 \to V$ is identified with $S\to V$;
  \item for $u\in D\setminus\{0\}$, $f_u: \tilde S_u \to V$ is a Lefschetz elliptic surface.
  \end{itemize}
\end{definition}
\begin{remark}
Such a deformation is sometimes also called a {\em splitting deformation} of the elliptic surface.
\end{remark}
\begin{remark}
Note that by Ehresmann's fibration theorem \parencite{Ehresmann_1951}, there is a trivialisation $\tilde S \simeq S\times V$. In particular this implies that for every $u\in D$, $\tilde S_u$ is diffeomorphic to $S$, but not necessarily biholomorphic.
\end{remark}
Morsifications of elliptic surfaces are useful as they allow the use of the results for Lefschetz fibrations to obtain information about general elliptic surfaces.
Indeed as $\eta$ is a trivial fibration, it induces an isometry $H_2(S_0)\simeq H_2(S_u)$ for all $u\in D$, and we may apply the results of \secref{sec:lefschetz} to describe the latter.
The existence of a morsification of any elliptic surface is guaranteed by the following theorem of \textcite{Moishezon77}.

\begin{theorem}[{\textcite[Thm. 8]{Moishezon77}}]\label{thm:morsification}
Let $S\to V$ be an elliptic surface. 
There exists a morsification of $S\to V$. 
Moreover, the number of singular fibres of $f_u\colon \tilde S_u\to V$ for $u\in D\setminus\{0\}$ does not depend on $u$.
\end{theorem}
\begin{remark}
Because we demand $f_u\colon \tilde S_u\to V$ to be a Lefschetz fibration for $u\ne 0$, the number of singular fibres is equal to the Euler characteristic of $S$, which explains why it has to be constant.
\end{remark}

We will apply this result to neighbourhoods of each critical value to obtain local morsifications.
For $c\in \Sigma$, define $D_c$ a closed disk around $c$ (i.e., continaing $c$ in its interior) in $\Proj^1$ such that $b\notin D_c$
Let $\ell_c$ be a path connecting $b$ to a point $b_c\in \partial D_c$.
Assume that for $c\ne c'$, $D_c\cap D_{c'} = \emptyset$ and the interior of $\ell_c$ and $\ell_{c'}$ do not intersect.
Let $\infty\notin \bigcup_{c\in\Sigma}D_c \cup \{b\}$, identify $\C = \Proj^1\setminus\{\infty\}$, and let $S^* = f^{-1}(\C)$.
Define $T_c = f^{-1}(D_c)$ and $F_{b_c} = f^{-1}(b_c)$.
Then we have the following lemma.
\begin{lemma}\label{lem:local_decomp}
The inclusion yields an isomorphism
\begin{equation}
\bigoplus_{c\in \Sigma} H_*(T_c, F_{b_c}) \to H_*(S^*, F_b)\,,
\end{equation}
where the identification $F_{b_c}\simeq F_b$ is given by $\ell_c$.
\end{lemma}
\begin{proof}
The main line of the argument is that the retraction of $\C$ to $\bigcup_{c\in\Sigma}D_c\cup \ell_c$ lifts through the fibration. For more details, see \textcite[\S5.3]{lamotke}.
\end{proof}

Notice that $T_c\to D_c$ is an elliptic surface with a single singular fibre.
In particular it is sufficient to study how morsifications behave on such elliptic surfaces.
This is the content of the next section.

\subsubsection{Local morsification}\label{sec:local_morsification}

Let $V\subset\Proj^1$ be a disk in $\Proj^1$ and consider an elliptic surface $S\to V$ with a single singular fibre.
Let $c$ denote the critical value and $b$ a regular value on the boundary of $V$.
Fix a symplectic basis of $H_1(F_b)$ and let $M_\infty$ be the monodromy matrix around $c$ in this basis.
From \thmref{thm:morsification}, there exists a morsification of $S\to V$
\begin{equation}
    \begin{tikzcd} [sep = .5 cm]
       \tilde S \arrow[r, "\tilde f"] \arrow[rr, dashed, bend right, "\eta"]& V\times D \arrow[r, "p"]& D
    \end{tikzcd}\,.
\end{equation}
Pick $u\ne 0$, and consider $S' = \tilde S_u = \eta^{-1}(u)$ with the Lefschetz fibration $f_u:S'\to V$.
Denote by $r$ its number of singular fibres and by $\Sigma$ its set of critical values.
We denote by $F'_v$ the fibre $f_u^{-1}(v)$.
\begin{lemma}\label{lem:splitting_fibres}
Following the terminology of \textcite{CadavidVelez08}, the number $r$ of singular fibres of $S'\to V$ is the number of factors in the minimal normal factorisation of $M_T$ (see \tabref{tab:KodairaClassification}).
Let $G_r\dots G_1$ be this factorisation.
Let $A\in \SL_2(\Z)$ be a matrix such that $M_\infty = AM_TA^{-1}$.
Then there is a distinguished basis of $\pi_1(D\setminus \Sigma, b)$ such that the corresponding monodromy matrices $M_1, \dots, M_r$ are given by
$M_i = AG_iA^{-1}$.
\end{lemma}
\begin{proof}
The first part is simply the observation that
\begin{equation}
r = \chi(S) = \chi(S') = \sum_{v\in\Sigma}\chi(F'_v)\,,
\end{equation}
and $\chi(F'_v)=1$ for every $v$ as $F'_v$ is of type $I_1$. 
For the second part, let $\ell_1, \dots, \ell_r$ be a distinguished basis of $\pi_1(D\setminus\Sigma, b)$.
Let $1\le i\le r-1$ and notice that the bases
\begin{equation}
(\ell_1, \dots, \ell_{i-1}, \ell_i\ell_{i+1}\ell_i^{-1}, \ell_i, \ell_{i+2}, \dots, \ell_r)
\end{equation}
and
\begin{equation}
(\ell_1, \dots, \ell_{i-1}, \ell_{i+1}, \ell_{i+1}^{-1}\ell_i\ell_{i+1}, \ell_{i+2}, \dots, \ell_r)
\end{equation}
are also distinguished bases.
Then the result is a direct application of \textcite[Thm.19]{CadavidVelez08}, as Hurwitz moves on the product $M_r\dots M_i$ can be achieved by the above changes of basis, see \eqref{eq:comp_monodromy}.
\end{proof}

\begin{remark}
As shown by \textcite[Thm. 21]{CadavidVelez08}, the choice of the factorisation of $M_T$ as a product of $I_1$ monodromy matrices does not matter. We could equivalently take any such minimal factorisation, such as the ones in \textcite[Table 5]{Naruki87}.
\end{remark}

Pick such a distinguished basis $\ell_1, \dots, \ell_r$ of $\pi_1(D\setminus \Sigma, b)$ and let $\Delta_1, \dots, \Delta_r$ be the corresponding thimbles.
Then the trivialisation of $\tilde S$ through $\eta$ yields an isomorphism $H_2(S, F_b) \simeq H_2(S', F'_b) = \bigoplus_{i=1}^r \Z\Delta_i$.
We conclude with two lemmas linking extensions and components of the singular fibre of $f:S\to V$ to this basis of thimbles.

\begin{lemma}\label{lem:inclusion_extensions}
Let $\ell_\infty$ be the simple loop around $c$ pointed at $b$. Then
\begin{equation}
\tau_{\ell_\infty} = \sum_{i=1}^r \tau_{\ell_i}\circ {\ell_{i-1}}_*\circ\cdots\circ{\ell_1}_*\,.
\end{equation}
\end{lemma}
\begin{proof}
For $u\in D$, let $\Sigma_u$ be the set of critical values of $f_u\colon S_u \to V$.
Define $\tilde\Sigma = \bigcup_{u\in D}\Sigma_u\times\{u\}$.
$\tilde\Sigma$ is an analytic set of $V\times D$ and the projection onto $D$ is a finite morphism of degree $r$, totally ramified at $c$.
Clearly, $\ell_\infty$ and $\ell_r\dots\ell_1$ have the same homotopy class in $\pi_1((V\times D)\setminus\tilde\Sigma)$.
The lemma is then a direct application of \eqref{eq:extensions}.
\end{proof}

\begin{lemma}\label{lem:inclusion_singular_components}
The inclusion of the components of $F_0$ in $H_2(S, F_b)$ coincides with the kernel of the boundary map:
\begin{equation}
\bigoplus_{i=1}^{m_c-1} \left\langle\Theta^c_i\right\rangle = \ker\left( \delta \colon  H_2(S', F'_b)\to H_1(F'_b) \right)\,.
\end{equation}
In particular, $r=m_c$ if $F_c$ has type $I_{m_c}$ and $m_c+1$ otherwise.
\end{lemma}
\begin{proof}
The direct inclusion is clear.
The rank of this kernel is~$r-1$ in the case of fibres of type $I_\nu$ and~$r-2$ in the other cases.
Therefore, if the mentioned equality holds, the last statement follows.

Let us detail the proof of the equality in the case of a fibre of type $I_3$.
From \lemref{lem:splitting_fibres}, its morsification splits it into $3$ fibres of type $I_1$, for each of which the monodromy matrix is (up to a global conjugation) $U$.
There are thus $3$ thimbles, the restriction of the boundary map to these thimbles has rank $1$, and the kernel thus has rank $3-1 =2$.
The intersection matrix of (the lift in $H_2(S)$ of) this kernel is 
\begin{equation}
\left(\begin{array}{rr}
-2 & -1 \\
-1 & -2
\end{array}\right)\,,
\end{equation}
and it is thus a sublattice of discriminant $3$.
As this coincides with the discriminant of the sublattice generated by the components of singular fibres \parencite{Kodaira63}, and as one is contained in the other, these sublattices are equal.

A similar direct computation gives the same result for each possible fibre type.

\end{proof}

An illustration of the effects of a morsification is provided in \figref{fig:morsification}.

\subsection{Global homology and periods}\label{sec:periods}\label{sec:homology_and_periods}

We are now ready to give an algorithm for computing $H_2(S)$ (with its lattice structure) from the action of monodromy on $F_b$ of the fibration $S\to \Proj^1$.

\begin{algorithm}
\caption{Homology of elliptic surface}\label{alg:homology}
\begin{algorithmic}
\Require the monodromy matrices $M_1, \dots, M_r \in \SL_2(\Z)$
\Ensure a description of $H_2(S)$ with its lattice structure
\State $N \gets [\hspace{.5em}]$
\For{$1\le i\le r$ decreasing}
    \State Find $A\in SL_2(\Z)$ and $T$ such that $M_i = AM_TA^{-1}$ \Comment{See \tabref{tab:KodairaClassification} for $M_T$}
    \For{$W$ in the minimal normal factorisation of $M_T$}
    \State Append $AWA^{-1}$ to $N$
    \EndFor
\EndFor
\For{$M_i$ in $N$}
	\State Compute $d_i\in \Z^{2\times 1}$ and $m_i\in \Z^{1\times 2}$ such that $M_i = I_2 + d_i m_i$.
	\State $T_i \gets (-1)^{n-1}   \left(\begin{array}{ccc} &\mathbf 0& \\\hline &m_i&\\\hline &\mathbf 0& \end{array}\right)$\,, where $m_i$ is the $i$-th line.
\EndFor
\State $B\gets \left(
    \begin{array}{c|c|c}
      & & \\
      d_1 & \dotsb & d_r \\
      & &
    \end{array}
  \right)$
\State $T_\infty \gets T_1 + T_2 M_1 + T_3 M_2 M_1 + \dotsb + T_r M_{r-1} \dotsb M_1$
\State $k \gets \ker B$
\State $i\gets \im T_\infty$
\State $H\gets k/i$

\noindent$H$ is identified to the subspace of $H_2(S^*)/H_2(F_b)$ generated by extensions. The (representatives of) vectors of $H$ are the coordinates in the basis of thimbles of $H_2(S^*, F_b)$. For more details, see \textcite[\S\S3,5]{LairezPichonVanhove23}.

\end{algorithmic}
\end{algorithm}

Note that per \lemref{lem:inclusion_singular_components} the sublattice generated by the components of a given singular fibre is explicitly identified in this description of homology; and per \lemref{lem:inclusion_extensions}, the extensions of $\T\cap \ker\delta$ are also explicitly identified.

To conclude this section, we provide a way to compute the periods of certain $2$-forms on this basis of homology.
Let $\omega\in H^2(S)$ and assume that $\omega = \omega_t\wedge \ud t$ for some $1$-form $\omega_t \in H^1(S)$, where $t$ denotes the dependance on a coordinate of $\Proj^1$.
Then the integral of $\omega$ on an extension can be obtained as a path integral of a period of the elliptic fibre via the observation that
\begin{equation}\label{eq:computing_periods}
\int_{\tau_\ell(\eta)}\omega = \int_{\ell}\left(\int_{\eta_t}\omega_t\right)\ud t\,,
\end{equation}
where $\eta_t\in H_1(F_t)$ is the unique deformation of $\eta\in H_1(F_b)$ along $\ell$.
In particular, it is possible to recover the periods of $\omega$ on $\T$ with high precision from the Picard--Fuchs equation of $\omega_t$ using numerical integration methods with quasilinear algorithmic complexity with respect to precision alone \textcite{VanDerHoeven1999, Mezzarobba2010}.
In practice, we rely on the implementation provided in SageMath by~\textcite{Mezzarobba2016} in the \emph{ore\_algebra} package \parencite{KauersJaroschekJohansson2015}.
For further details on the computation of periods on thimbles, see \textcite[\S3.7]{LairezPichonVanhove23}.

As the fibre components are localised in a single fibre, the periods of $\omega$ on a fibre component are all zero.
The following lemma shows that this information, i.e., the periods on extension and fibre components, is sufficient to recover the full period mapping.

\begin{lemma}\label{lem:fullrank}
The primary lattice $\Prim(S/\Proj^1)$ has full rank.
\end{lemma}
\begin{proof}
For each $c\in \Sigma$, $H_2(T_c, F_{b_c})_\Q = \left(\ker\delta_c\right)_\Q\oplus \left(\im\tau_c\right)_\Q$.
Furthermore, $\T = \bigoplus_{c\in \Sigma} \im\tau_c$.
For a {$\Z$-module} $A$, denote by $A_\Q$ the tensor product $A\otimes\Q$.
From \lemref{lem:inclusion_singular_components} and  \lemref{lem:local_decomp}, we have that
\begin{equation}
\begin{split}
\left(\T \cap\ker\delta\right)_\Q &\oplus \bigoplus_{\substack{c\in\Sigma\\ 1\le i \le m_c-1}}\langle\Theta^c_i\rangle_\Q \oplus \langle O\rangle\\
&= \left(\ker\delta\right)_\Q \cap \bigoplus_{c\in\Sigma} \left(\im\tau_c\right)_\Q\oplus \left( \ker\delta_c\right)_\Q\,,\\
&= \left(\ker\delta\right)_\Q\,.
\end{split}
\end{equation}
\lemref{lem:quotient_kerd} allows to conclude.
\end{proof}

Let $\mathcal B = (\eta_1, \dots, \eta_s, F_b, O)$ be the basis of $H_2(S)$ obtained from \algoref{alg:homology} and let 
\begin{equation}
\mathcal B' = (\Gamma_1, \dots, \Gamma_k, \Theta_1, \dots, \Theta_{s-k}, F_b, O)
\end{equation}
be a basis of $\Prim(S/\Proj^1)$, such that for each $i$, $\Gamma_i \in \T\cap\ker\delta$ is an extension and $\Theta_i \in \ker\delta_c$ for some $c\in \Sigma$ is a fibre component.
Over $\Q$, $\mathcal B'$ is also a basis of $H_2(S)_\Q$, and the matrix of change of basis $M_{\mathcal B'\to \mathcal B}\in \GL(\Q)$ has integer coefficients and can be computed with \lemref{lem:inclusion_extensions} and \lemref{lem:inclusion_singular_components}.

Using the methods of \textcite[\S3.7]{LairezPichonVanhove23}, we may numerically compute the periods $\int_{\Gamma_i}\omega$.
Assume $\omega$ is a holomorphic form\footnote{This assumption may be dropped in general. Indeed $\mathrm dt$ restricts to $0$ on fibres and $\omega_t$ on sections, so that $\omega$ evaluates to $0$ whenever integrated on these cycles.}. Then, as the integration cycles of the periods $\int_{\Theta_i}\omega$, $\int_{F_b}\omega$ and $\int_{O}\omega$ are algebraic, the periods are zero by Lefschetz's $(1,1)$ theorem (see \thmref{thm:lefschetz11} below).
Let 
$\pi_{\mathcal B} = \left(\int_\gamma \omega \right)_{\gamma\in\mathcal B}$
be the vector of periods of $\omega$ on the basis $\mathcal B$, and similarly for $\mathcal B'$.
Then
\begin{equation}
\pi_{\mathcal B'} = \left(\int_{\Gamma_1}\omega, \dots, \int_{\Gamma_t}\omega,0,\dots, 0\right)
\end{equation}
and
\begin{equation}
\pi_{\mathcal B} = M_{\mathcal B'\to \mathcal B}^{-1}\pi_{\mathcal B'}\,.
\end{equation}

\section{Application: Mordell--Weil group and lattice}
In this section, we explain how to compute explicit embeddings of the Néron--Severi lattice in the description of $H_2(S)$ given in the previous section.
We then use this to recover the Mordell--Weil group, and the lattice structure of its torsion-free part, the Mordell--Weil lattice.

We start by recalling generalities about these lattices.

\begin{definition}
The {\em Néron--Severi lattice} $\NS(S)$ is the sublattice of $H_2(S)$ generated by classes of divisors. Its rank is called the {\em Picard rank} or {\em Picard number} or {\em Néron--Severi rank}.
\end{definition}

\begin{definition}\label{def:trivial}
The {\em trivial lattice} $\Triv(S/\Proj^1)$ is the sublattice of $\NS(S)$ generated by the zero section and the fibre components.
Its orthogonal complement is the {\em essential lattice} $L(S/\Proj^1) = \Triv(S/\Proj^1)^\perp$.
\end{definition}

\begin{definition}
The {\em Mordell--Weil group} $E(\C(t))$ of the elliptic curve $E/\C(t)$ is the group of its $\C(t)$-rational points.
\end{definition}

As mentioned in the beginning of \secref{sec:elliptic_surfaces}, rational points of $E(\C(t))$ are in bijection with sections of $S\to \Proj^1$, sending a point $P\in E(\C(t))$ to its Zariski closure $\bar P$.
The following lemma shows that the group structure of $E(\C(t))$ coincides with the lattice structure of $H_2(S)$ modulo the trivial lattice.

\begin{theorem}[{\textcite[Thm 6.5]{SchuttShioda10}}]
The map $P\mapsto \bar P \mod \Triv(S/\Proj^1)$ is an isomorphism from $E(\C(t))$ to $\NS(S)/\Triv(S/\Proj^1)$.
\end{theorem}

In particular, this equips the torsion-free part of the Mordell--Weil group with a lattice structure, inherited from the lattice structure on $\NS(S)\subset H_2(S)$.
More precisely, the orthogonal projection of $\NS(S)_\Q = \NS(S) \otimes \Q$ onto $L(S/\Proj^1)_\Q = L(S/\Proj^1) \otimes \Q$ defines a map ${\phi\colon  E(\C(t))\to L(S/\Proj^1)_\Q}$.
The kernel of this map is the torsion subgroup, and thus $\phi$ equips $E(\C(t))/E(\C(t))_{tor} $ with a rational lattice structure.

\begin{definition}
The {\em Mordell--Weil lattice} $\operatorname{MWL}(S)$ of $S/\Proj^1$ is the resulting lattice $-E(\C(t))/E(\C(t))_{tor}$. 
\end{definition}

\begin{remark}
The minus sign means that the bilinear pairing is the opposite of the one induced naturally. 
This is chosen so that $\operatorname{MWL}(S)$ is a positive-definite lattice, see \textcite[\S11]{SchuttShioda10}.
\end{remark}

For further reading on this topic, we recommend \textcite{SchuttShioda10}.

\subsection{Computing the Néron--Severi lattice}\label{sec:periods}
The first step toward computing the Mordell--Weil lattice is to compute the Néron--Severi lattice $\NS(S)$.
By Lefschetz's $(1,1)$ theorem \textcite[\S1.2]{GriffithsHarris78}, it is entirely characterised as the kernel of the holomorphic period mapping.
\begin{theorem}[Lefschetz (1,1) theorem]\label{thm:lefschetz11}\label{thm:lefschetz11}
Let $\omega_1, \dots, \omega_s$ be a basis of the space of holomorphic $2$-forms of $S$, $H^{2,0}(S)$, and consider the period map
$\pi\colon  H_2(S)\to \C^s, \gamma\mapsto (\int_\gamma\omega_1, \dots, \int_\gamma\omega_s)$. Then
\begin{equation}
\NS(S) = \ker\pi\,.
\end{equation}
\end{theorem}
We compute this kernel heuristically using the LLL method. 
In order to do this we need two things: a basis of $H^{2,0}(S)$ and numerical approximations of the associated periods.

Let $\omega$ be a holomorphic $2$-form on $S$. 
As an element of $H^0(S, \Omega_S^2)$, it can be written as 
\begin{equation}
\omega = f(t) \omega_t \wedge \ud t\,,
\end{equation}
where $\omega_t \in H^1(E)$ is a rational section of the holomorphic $1$-form bundle of the generic fibre, and $f\in \Q(t)$ is a rational function.
This representation is well adapted to the integration algorithm of \textcite[\S3.7]{LairezPichonVanhove23}.
Of course the converse is not true: not every rational function will yield a holomorphic $2$-form on $S$. 
In fact the rational functions for which this is true are very tightly controlled by the \emph{Picard--Fuchs equation} $\Lambda$ of $\omega_t$ --
that is, the minimal differential equation satisfied by $\omega_t$ with respect to the connection inherited from the derivation on $\C(t)$ through the fibration over $\Proj^1$.

Before focusing on the rational functions, let us briefly recall how $\omega_t$ and its Picard--Fuchs equation can be computed. 
Let $P_t$ be the defining equation of the elliptic fibration, given as a polynomial in $\Q[t, X,Y,Z]$ that is homogeneous of degree $3$ in $X,Y,Z$.
In particular, the defining equation for $F_a$ is $P_a$ whenever $a$ is regular.
\begin{proposition}
There is a natural isomorphism $\res\colon  H^2(\Proj^2\setminus F_a)\to H^1(F_a)$, called the {\em residue mapping}.
\end{proposition}
Since $\Proj^2\setminus F_a$ is affine, its De Rham cohomology can be computed using algebraic forms directly \textcite{Grothendieck1966}.
More precisely, it is the quotient of the space of homogeneous rational fractions of the form $\frac{A}{P_a^k}$ with degree $-3$ by derivatives of the same form.
Furthermore, the holomorphic form of $H^1(F_a)$ is identified with the residue of the image of $\frac{1}{P_a}$ in this description \parencite[(8.6)]{Griffiths1969a}.
All in all, $\omega_t$ can be taken to be $\res \frac{1}{P_t}$.
Its Picard--Fuchs equation can then be computed using Griffiths--Dwork reduction \parencite[\S 5.3]{CoxKatz1999}.
For a more detailed discussion, see \textcite[\S3.2]{LairezPichonVanhove23}.

We now turn back to the computation of valid rational coefficients.
Let us first recall some standard definitions.
\begin{definition}
Let $y_1, y_2$ be a basis of solutions of a second order differential equation $\Lambda = a(t)\partial_t^2 + b(t)\partial_t + c(t)$.
Then \emph{the Wronskian} $W$ of $\Lambda$ is $W(t) = y_1(t)y_2'(t) - y_2(t)y_1'(t)$.
It is the solution to the differential equation $a(t)\partial_t - b(t)$ and is defined up to a scalar.
When $\Lambda$ is the Picard--Fuchs equation of an elliptic surface, we have $W(t) \in \Q(t)$ \parencite[Theorem II.2.5]{Stiller_1981}.
\end{definition}
\begin{definition}
Let $\frak{A}$ be a divisor on $\Proj^1$. The \emph{linear system} $\mathcal{L}(\frak{A})$ associated to $\frak{A}$ is the space of rational functions on $\Proj^1$ for which the pole order at a point $p\in \Proj^1$ is at most $\operatorname{ord}_p\frak{A}$.
\end{definition}
The following result of \textcite{Stiller87} gives a way to compute rational functions $f_1, \dots, f_r\in \Q(t)$ such that:
\begin{itemize}
\item $f_i(t)\omega_t\wedge \ud t$ defines a holomorphic $2$-form on $S$;
\item and $f_1(t)\omega_t\wedge \ud t, \dots, f_r(t)\omega_t\wedge \ud t$ is a basis of $H^{2,0}(S)$.
\end{itemize}
\begin{theorem}[{\textcite[\S3]{Stiller87}}]\label{thm:holomorphic_forms}
There is a divisor $\frak{A}_0$ on $\Proj^1$ depending only on the Picard--Fuchs equation $\Lambda$ such that if $Z\in\frak{A}_0$, then $\frac{Z}{W}\omega_t\wedge \ud t$ is a holomorphic $2$-form on $S$, where $W\in\Q(t)$ is the Wronskian of $\Lambda$. 
Furthermore the map $\mathcal{L}(\frak{A}_0) \to H^{2,0}(S)$ is an isomorphism.
\end{theorem}
%

An algorithm for computing $\frak{A}_0$ is given in \textcite[\S3]{Stiller87}, and we recall it here for the sake of completeness.
Assume that the Picard--Fuchs equation $\Lambda$ has order two\footnote{This may fail when $S$ is an isotrivial elliptic surface in which case the Picard--Fuchs equation has order $1$. In that case, the author does not know of a way to identify the holomorphic form.}, or in other words that $\omega_t$ and its derivative generate the full space of $1$-forms of the underlying elliptic curve.
In particular at any point $p\in\Proj^1$, the space of local solutions in a slit neighbourhood of $p$ is generated by two solutions, that are locally of the form 
\begin{equation}
(t-p)^{q}(h_1(t-p)\log(t-p) +h_2(t-p))\,,
\end{equation}
with $q\in \Q$, $h_1$ and $h_2$ two holomorphic functions in a neighbourhood of $0$. Let $r\le s$ be the respective leading exponents of these two solutions.
Then the order of $\mathfrak{A}_0$ at $p$ is 
\begin{equation}
	\operatorname{ord}_p \mathfrak{A}_0 = \begin{cases} 
	-\lfloor s\rfloor -3\text{ if }p=\infty\\
	-\lfloor s\rfloor +1 \text{ otherwise}
	\end{cases}
\end{equation}
In particular, if $p$ is a finite regular point of $\Lambda$, then we obtain $\operatorname{ord}_p \mathfrak{A}_0 = 0$, meaning we can consider solely the singular points of $\Lambda$.
For further reading on the topic, we recommend \textcite{Stiller87} and \textcite[\S4.3]{DoranKostiuk2019}.

The local exponents of $\Lambda$ can be obtained symbolically (see \textcite{Frobenius1873} and \textcite[\S4]{Mezzarobba2010} for instance), and we thus have a method to compute a basis of the holomorphic $2$-forms of $S$, with a presentation that is well suited for the integration methods of the periods mentioned in \secref{sec:homology_and_periods}.
We can thus compute high precision numerical approximations of the holomorphic periods of $S$.
The Néron--Severi group can be heuristically computed by recovering integer linear relations between these periods. Indeed, let $\alpha_i$ be integers. Then
\begin{equation}
	\int_{\sum_i\alpha_i\gamma_i}\omega =\sum_i\alpha_i\left(\int_{\gamma_i}\omega\right) =0 \text{ for all }\omega\in H^{2,0}(S)
\end{equation}
if and only if the cycle $\sum_i\alpha_i\gamma_i\in \NS(S)$, where the $\gamma_i$'s form a basis of $H_2(S)$.
Thus integer linear relations between the holomorphic period vectors $(\int_{\gamma_i}\omega_1, \dots, \int_{\gamma_i}\omega_s)$ are in bijection with $\NS(S)$.

In order to recover these linear relations, we use the LLL algorithm. 
This computation is not certified, and may fail in two ways: the algorithm may miss integer relations with large coefficients, or may recover ``fake'' linear relations that hold up to very high precision. 
More precisely, the algorithm provides a sublattice $\Lambda\subset H_2(X)$ and positive numbers $B$, $N$ and $\varepsilon$ that depend on precision such that
\begin{enumerate}
\item $\Lambda = \NS(X)$; or
\item $\NS(X)$ is not generated by elements of the form $\sum_{i}\alpha_i\gamma_i$ with $ \sum_i \alpha_i^2 \le B$; or
\item There exists $\sum_{i}\alpha_i\gamma_i \notin \NS(X)$ such that 
\begin{equation}
    \sum_j \left|\sum_i\alpha_i\int_{\gamma_i}\omega_j\right|^2 \le \varepsilon^2 \hspace{2em} \text{and} \hspace{2em}  \sum_i\alpha_i^2 \le N^2\,.
\end{equation}
\end{enumerate}
In practice, for $300$ recovered decimal digits of precision for the periods of an elliptic $K3$ surface (which was obtained in a few seconds in all the cases that we tried), we find $B \simeq 10^{132}$, $N=3$, and $\varepsilon \simeq 10^{-271}$.
For further discussion on these issues, see \textcite{LairezSertozPicardRank}.

\subsection{Example: an elliptic curve with high Mordell--Weil rank over $\Q$}
In this section we detail the workings of our algorithm on an explicit example.
A SageMath worksheet reproducing the results mentioned here is available at \mbox{\emph{example\_paper\_elliptic.ipynb}}\footnote{\url{https://nbviewer.org/urls/gitlab.inria.fr/epichonp/eplt-support/-/raw/main/example_paper_elliptic.ipynb}}.
The elliptic surface $S/\Proj^1$ we consider is an elliptic K3 with Picard rank 19 used in \textcite[\S9]{ElkiesKlagsbrun2020} (with $u=5$ in their notations) to find the elliptic curve with highest known Mordell--Weil rank over $\Q$ for which the Mordell--Weil torsion subgroup is $\Z/2\Z$.
Its defining equation is
\begin{equation}
X^{3} + 4A(t)X^{2} Z  + 512 B(t) X Z^{2}-Y^{2} Z
\end{equation}
where
\begin{equation}
A(t) = 93273 t^{4} + 58840 t^{3} + 102618 t^{2} + 35680 t + 14485
\end{equation}
and
\begin{equation}
\begin{split}
B(t) = -8590032 t^{8} - 78412620 t^{7} + 17011856 t^{6}
+241822775 t^{5} - 19459741 t^{4} - 127136490 t^{3} \\
+16161642 t^{2} + 15406335 t - 2083725
\end{split}
\end{equation}

\subsubsection*{The homology lattice of $S$}

This elliptic fibration has $16$ singular fibres above points $c_1, \dots, c_{16}$. 
We pick a basepoint $b$ as well as a distinguished basis $\ell_1,\dots, \ell_{16}$ of ${\pi_1(\C^1\setminus\{c_1, \dots, c_{16}\}, b)}$.
The corresponding monodromy matrices in a chosen symplectic basis~$\gamma_1, \gamma_2$ of the homology of the fibre are given by
\begin{equation}
\begin{gathered}
M_1 = 
\left(\begin{array}{rr}
7 & 9 \\
-4 & -5
\end{array}\right)\,,\hspace{1em}
M_i = 
\left(\begin{array}{rr}
1 & 1 \\
0 & 1
\end{array}\right)
\text{ for $i = 2,3,15$}\,\\
M_i =
\left(\begin{array}{rr}
3 & 1 \\
-4 & -1
\end{array}\right)
\text{ for $i = 4,9,11,12$, and }\\
M_i=
\left(\begin{array}{rr}
3 & 2 \\
-2 & -1
\end{array}\right)
\text{ for $i = 5,6,7,8,10,13,14,16$.}
\end{gathered}
\end{equation}

Computing the $\SL_2(\Z)$ conjugacy class of these matrices, one finds that eight fibres (those for which the monodromy matrix is $M_5$) are $I_2$ fibres, and the remaining eight are Lefschetz, i.e., of type $I_1$.
Indeed, we have 
\begin{equation}
M_5 = AM_{I_2}A^{-1} \text{ with } A = 
\left(\begin{array}{rr}
1 & 0 \\
1 & 1
\end{array}\right)
\in\GL_2(\Z)\,.
\end{equation}
Per the minimal normal factorisation of \tabref{tab:KodairaClassification}, $M_{I_2} = U^2$.
Therefore in a morsification $S'/\Proj^1$ of $S/\Proj^1$, there is a distinguished basis $\ell'_1,\dots, \ell'_{24}$ of $\pi_1(\C \setminus\Sigma', b)$ consisting of $8+8\times 2 = 24$ elements, where $\Sigma'$ is the set of critical values of the morsification, and such that the associated monodromy matrices are given by
\begin{equation}
\begin{gathered}
M'_1 = 
\left(\begin{array}{rr}
7 & 9 \\
-4 & -5
\end{array}\right),\hspace{1em}
M'_i = 
\left(\begin{array}{rr}
1 & 1 \\
0 & 1
\end{array}\right)
\text{ for $i = 2,3,22$},\\
M'_i =
\left(\begin{array}{rr}
3 & 1 \\
-4 & -1
\end{array}\right)
\text{ for $i = 4,13,16,17$, and }\\
M'_i=
\left(\begin{array}{rr}
2 & 1 \\
-1 & 0
\end{array}\right) =
AUA^{-1}
\text{ for all other $i$.}
\end{gathered}
\end{equation}

We may then use the methods of \textcite[\S\S3,5]{LairezPichonVanhove23} to compute an effective basis $\Gamma_1, \dots, \Gamma_{22}$ of the homology of $S'$ in terms of the thimbles $\Delta'_1, \dots, \Delta'_{24}$, the fibre class and the zero section.
For instance, we find that a non trivial homology class is given by $\Delta'_2 - \Delta'_{22}$:
as~$M'_2 = M'_{22}$, this relative homology class has empty boundary and thus lifts to a class in $H_2(S')$. 
It is non-trivial as $\ell'_{22}{\ell'}^{-1}_{2}$ is non-trivial in $\pi_1(\Proj^1\setminus \Sigma')$.

With this description, the components of an $I_2$ fibre of $S$ above a critical value $c$ can be obtained as the kernel of the thimbles of critical points flowing together at $c$.
More explicitly, the component of the fibre above $c_5$ is the homology class corresponding to the lift of $\Delta'_5-\Delta'_6$.

Furthermore, extensions of $S$ can also be described in this basis.
For example
\begin{equation}
\tau_{\ell_6^{-1}\ell_5}(\gamma_2) = \tau_{\ell_8^{-1}\ell_7^{-1}\ell_6\ell_5}(\gamma_2) = \Delta_5 +\Delta_6 - \Delta_7 - \Delta_8\,.
\end{equation}
All in all we obtain the coordinates of a basis of $\Prim(S/\Proj^1)$ in the basis of $H_2(S)$ obtained from the morsification $S'$.
From the Picard--Fuchs equation of the surface, which has order $2$ and degree $26$, we recover using \thmref{thm:holomorphic_forms} that the space of holomorphic forms of $S$ is generated by
\begin{equation}
\omega = \res\frac{1}{P_t}\wedge \ud t\,.
\end{equation}
From then on, we can compute the periods on the primary lattice, and recover the full period mapping using the coordinates computed above.
For example, we find that the holomorphic period of the first element of the basis of homology we computed is 
\begin{equation}
\int_{\Gamma_1} \omega = -0.0007064447191\cdots  - i0.0002821239749\cdots \,,
\end{equation}
with certified precision bounds of around 150 digits.

Using the LLL algorithm, we find that the Néron--Severi lattice has rank $19$ as expected.
Finally the Mordell--Weil group is obtained as the quotient of the Néron--Severi lattice by the trivial lattice.
We find
\begin{equation}
\MW(S)\simeq \Z^9 \times \Z/2\Z\,.
\end{equation}

It should be noted that the result of \textcite[\S9]{ElkiesKlagsbrun2020} is stronger than what we have computed here. 
First our approach for computing the Néron--Severi group is heuristic as it relies on the LLL algorithm --- in particular it is not a certified computation, and thus does not provide a proof.
Secondly we have merely shown a result for the Mordell--Weil group over $\C(t)$\footnote{or, rather, over $\Qbar(t)$.}, whereas \textcite{ElkiesKlagsbrun2020} shows that this computation holds over $\Q$ for one of the fibres.

\section*{Acknowledgements}
I would like to thank Charles Doran, Pierre Lairez, Erik Panzer, Duco van Straten, and Pierre Vanhove for valuable and insightful discussions.
I am also thankful to the reviewers for their helpful and constructive feedback which elevated the quality of this text.

\section*{Funding}{This work has been supported by the Agence nationale de la recherche
  (ANR), grant agreement ANR-19-CE40-0018 (De Rerum Natura), grant agreement
ANR-20-CE40-0026-01 (Symmetries and moduli spaces in algebraic geometry and physics); and the European
  Research Council (ERC) under the European Union’s Horizon Europe research and
  innovation programme, grant agreement 101040794 (10000~DIGITS)}

\bibliographystyle{elsarticle-harv}
\bibliography{main}

\end{document}

%% file: figs/homotopy_basis/homotopy_basis.tex
\newcommand{\ellpath}[5] 
{
	\draw (#1) node {$\bullet$} node[above] {#3}  [decoration={markings, mark=at position 1 with {\arrow{>}}}, postaction={decorate}] circle[radius=1] ($(#1)+(1,1)$) node {#4};
	\draw (#2) to[bend right, #5]  node[midway, below, sloped] {$\longrightarrow$} node[midway, above, sloped] {$\longleftarrow$} ($(#1)+(0,-1)$;)
}
\fontsize{15}{12}\selectfont
\begin{tikzpicture}

	\coordinate (tr) at (-1,8);
	\coordinate (t1) at (10,6);
	\coordinate (t2) at (7,8);
	\coordinate (b) at (0,3);
	\coordinate (dots) at (3,7);
	
	\ellpath{t1}{b}{$t_1$}{$\ell_1$}{bend right};
	\ellpath{t2}{b}{$t_2$}{$\ell_2$}{bend right};
	\ellpath{tr}{b}{$t_r$}{$\ell_r$}{bend left};
	
	\draw (b) node {$\bullet$} node[below] {$b$};
	
	\draw (dots) node {$\dots$};

	\coordinate (P) at (-1,8);
	\coordinate (L) at (10,6);
	
\end{tikzpicture}

%% file: figs/thimbles/thimbles.tex
\begin{tikzpicture}

	\coordinate (t) at (3,-2);
	\coordinate (x) at (3,1);
	\coordinate (b) at (0,-2);

        	\draw[very thick] (0,0) ellipse (0.25 and 0.5);
	\draw (-0.25, 0) node[left] {$\gamma$};
        	\draw[very thick] (0,2) ellipse (0.25 and 0.5);
	\draw (-0.25, 2.5) node[left] {$\ell_{i*}\gamma$};
	
	\draw[fill=gray!20] (4,1) edge[in=0, out=-125] (0,0.5) edge[in=0, out=125] (0,1.5);
	\draw (5,1) to[in=0, out=-90] (3,-0.7) to[in=0, out=180] (0,-0.5);
	\draw (5,1) to[in=0, out=90] (3,2.7) to[in=0, out=180] (0,2.5);
	
	\draw (3,2.2) node {$\Delta_i = \tau_{\ell_i}(\gamma)$};

	\draw[dashed] (x) node {$\bullet$} node[left] {$x_i$} -- (t) node {$\bullet$} node[left] {$c_i$};
	
	\draw (b) node {$\bullet$} node[left] {$b$};
	\draw (b) --node[midway, below, sloped] {$\longrightarrow$} ($(t)-(0,0.5)$) arc (-90:90:0.5 and 0.5)  --node[midway, above, sloped] {$\longleftarrow$} (b);
	\draw ($(t)+(0.5,0)$) node[right] {$\ell_i$};

\end{tikzpicture}

%% file: figs/morsification_2/morsification_2.tex
\fontsize{20}{12}\selectfont
\begin{tikzpicture}
        
        \coordinate (b) at (4,-8);
        
        \coordinate (c) at (4,0);
        
        \draw[draw=black!70] (4,0) circle (8);
        \draw (7.5, 4) node {$V$};
        
        \draw[fill] (b) node[below] {$b$} circle (0.07);
        
        \draw[fill] (c) node[above left] {$c$} circle (0.07);
        
        \draw (b) .. controls ++(1,2) and ++(0,-5) .. ($(c) + (3,0)$) arc(0: 180: 3) node[midway, above]{$\textstyle\ell_\infty$} .. controls ++(0,-5) and ++(-1,2) ..  (b);
        
\end{tikzpicture}

%% file: figs/morsification/morsification.tex
\fontsize{20}{12}\selectfont
\begin{tikzpicture}
        
        \coordinate (b) at (4,-8);
        
        \coordinate (t3) at (4,0);
        \coordinate (t1) at (7.8,0);
        \coordinate (t2) at (6.6,0.6);
        \coordinate (t4) at (1.8,2.2);
        \coordinate (t5) at (0.8,-3);
        
        \draw[draw=black!70] (4,0) circle (8);
        
        \draw[fill] (b) node[below] {$b$} circle (0.07);

        \draw[fill] (t1)  circle (0.07);
        \draw[fill] (t2)  circle (0.07);
        \draw[fill] (t3)  circle (0.07);
        \draw[fill] (t4)  circle (0.07);
        \draw[fill] (t5)  circle (0.07);
        
        \draw (b) .. controls ++(2,2) and ++(0,-5) .. ($(t1) + (0.5,0)$) arc(0: 180: 0.5) node[midway, above]{$\textstyle\ell_1$} .. controls ++(0,-5) and ++(2,2) ..  (b);
        \draw (b) .. controls ++(1,2) and ++(0,-5) .. ($(t2) + (0.5,0)$) arc(0: 180: 0.5) node[midway, above]{$\scriptsize\ell_2$} .. controls ++(0,-5) and ++(1,2) ..  (b);
        \draw (b) .. controls ++(0,2) and ++(0,-5) .. ($(t3) + (0.5,0)$) arc(0: 180: 0.5) node[midway, above]{$\scriptsize\ell_3$} .. controls ++(0,-5) and ++(0,2) ..  (b);
        \draw (b) .. controls ++(-1,2) and ++(0,-5) .. ($(t4) + (0.5,0)$) arc(0: 180: 0.5) node[midway, above]{$\scriptsize\ell_4$} .. controls ++(0,-5) and ++(-1,2) ..  (b);
        \draw (b) .. controls ++(-2,2) and ++(0,-3) .. ($(t5) + (0.5,0)$) arc(0: 180: 0.5) node[midway, above]{$\scriptsize\ell_5$} .. controls ++(0,-3) and ++(-2,2) ..  (b);
        
        \draw (b) .. controls ++(5,2) and ++(0,-3) .. (10,0) arc(0: 180: 6) node[midway, above]{$\scriptsize\ell_\infty$} .. controls ++(0,-3) and ++(-5,2) ..  (b);
        
\end{tikzpicture}